\documentclass[12pt]{amsart}
\usepackage[utf8]{inputenc}
\usepackage{graphicx}
\usepackage{epstopdf}
\usepackage{inputenc}
\usepackage{enumitem}
\usepackage{amsmath,fullpage}
\usepackage{amssymb}
\usepackage{emptypage}
\usepackage{mathtools}
\usepackage{float}
\usepackage[dvipsnames,x11names,svgnames]{xcolor}
\usepackage[colorlinks=true,linkcolor=NavyBlue,citecolor=DarkGreen, urlcolor=blue]{hyperref}
\usepackage{dsfont}

\newtheorem{theorem}{theorem}[section]
\newtheorem{exa}[theorem]{Example}
\newtheorem{thm}[theorem]{Theorem}
\newtheorem*{theorem*}{Theorem}
\newtheorem{defn}[theorem]{Definition}
\newtheorem{proposition}[theorem]{Proposition}

\newtheorem{corollary}[theorem]{Corollary}
\newtheorem{lemma}[theorem]{Lemma}

\newtheorem*{acknowledgements*}{Acknowledgements}
\newtheorem{remark}[theorem]{Remark}

\newcommand{\K}{\mathbb{K}}

\makeatletter
\def\blfootnote{\gdef\@thefnmark{}\@footnotetext}
\makeatother

\def\house#1{\setbox1=\hbox{$\,#1\,$}%
\dimen1=\ht1 \advance\dimen1 by 2pt \dimen2=\dp1 \advance\dimen2 by 2pt
\setbox1=\hbox{\vrule height\dimen1 depth\dimen2\box1\vrule}%
\setbox1=\vbox{\hrule\box1}%
\advance\dimen1 by .4pt \ht1=\dimen1
\advance\dimen2 by .4pt \dp1=\dimen2 \box1\relax}

\begin{document}
	\title[]{On Groups of Linear Fractional Transformations Stabilizing Finite Sets of Four Elements}

\author[P. Nyadjo Fonga]{{Patrick Nyadjo Fonga}}
\address{Bowling Green State university } 
\email{ nyadjop@bgsu.edu }

\begin{abstract}
Let \( E \) be a subset of the projective line over a commutative field \(\mathbb{K}\). When \(\mathbb{K}\) has infinite cardinality, it is well known that if \(E\) contains at most three elements, then the group of linear fractional transformations preserving \(E\) is either infinite or isomorphic to the symmetric group on three elements. In this work, we investigate the case where \(E\) consists of four elements. We show that the group of projective linear transformations stabilizing \(E\) is, depending on the characteristic of the field \(\mathbb{K}\), isomorphic to either the Klein four-group \(V_4\), the dihedral group \(D_4\) of order eight, the alternating group \(\mathfrak{A}_4\) of order twelve, or the symmetric group \(\mathfrak{S}_4\) of order twenty-four.

\medskip


\end{abstract}

\maketitle

\section{Introduction}

Let \( E \) be a vector space of dimension \( n + 1 \) over a field \(\mathbb{K}\), and throughout this paper, we assume that \(\mathbb{K}\) is commutative. The projective space associated with \( E \), denoted by \( \mathbb{P}(E) \), is defined as the set of all one-dimensional linear subspaces of \( E \). 

In the particular case where \( V = \mathbb{K}^2 \), the associated projective space \( \mathbb{P}(V) \) is called the \emph{projective line} over \(\mathbb{K}\), and is denoted by \( \mathbb{P}^1(\mathbb{K}) \). This projective line can be identified with the set \( \mathbb{K} \cup \{\infty\} \), where \( \infty \) represents a point not belonging to \(\mathbb{K}\).

If \( \mathbb{P}(V) \) and \( \mathbb{P}(W) \) are projective spaces of the same dimension over the same field \(\mathbb{K}\), then any vector space isomorphism \( f: V \to W \) induces a bijection from \( \mathbb{P}(V) \) to \( \mathbb{P}(W) \), preserving the incidence structure. This induced map is called a \emph{homography} or a \emph{projective transformation}. 

In the case of the projective line \( \mathbb{P}^1(\mathbb{K}) \), homographies correspond to \emph{linear fractional transformations} (or \emph{homographic functions}), which are functions \( \mathbb{P}^1(\mathbb{K}) \to  \mathbb{P}^1(\mathbb{K}) \) of the form
\[
z \mapsto \begin{cases} 
	\dfrac{az + b}{cz + d} \ \ \text{for } z \in \mathbb{ K} \text{  and  } cz+d \neq 0, \\
	\infty \ \ \text{for } z \in \mathbb{ K} \text{  and  } cz+d = 0, \\
	a/c \ \ \text{for } z = \infty \text{  and  } c \neq 0, \\
	\infty \ \ \text{for } z =\infty \text{  and  } c = 0.
\end{cases},\text{with } a, b, c, d \in \mathbb{K} \text{ and } ad - bc \neq 0.
\]
These transformations form a group under composition, denoted by \( \mathrm{PGL}_2(\mathbb{K}) \), the projective general linear group.\\
If \(\K\) is different from the binary field \(\mathbb{F}_2\), then for any \( \lambda \in \mathbb{K} \) with \( \lambda \neq 0 \) and \( \lambda \neq 1 \), we define the group \( \mathcal{G}_\lambda \) of homographies stabilizing the subset \( \{\infty, 0, 1, \lambda\} \) of \(\mathbb{P}^1(\mathbb{K}) \). More generally,
Given a subset \( E \subset \mathbb{P}^1(\mathbb{K}) \), we define the stabilizer subgroup
\[
G_E = \{ h \in \mathrm{PGL}_2(\mathbb{K}) \mid h(E) = E \},
\]
called the group of homographies associated with \( E \). When \( \K \) is infinite and \( E \) contains at most three elements, it is well known that \( G_E \) is either infinite or isomorphic to the symmetric group \( \mathfrak{S}_3 \); further details are given in Remark~\ref{remk1}.

In this paper, we focus on the case where \( E \) consists of exactly four elements. Our main result describes the structure of \( G_E \) in this situation, depending on the characteristic of the field \(\mathbb{K}\). It is important to note for what follows that when the characteristic of \( \K \) is \( 3 \), the polynomial \( X^2 + X + 1 \) splits in \( \K[X] \) as \( (X - 1)^2 \). If the characteristic is different from \( 3 \) and the polynomial \( X^2 + X + 1 \) splits in \( \K[X] \), we denote by \( j \) and \( j^2 \) its roots in \( \K \). These roots are distinct and satisfy \( j^3 = 1 \) and \( j \ne 1 \).

\begin{thm} \label{main}
Let \(\mathbb{K}\) be a field different from \(\mathbb{F}_2\), and let \(\lambda \in \mathbb{K}\) with \(\lambda \neq 0\) and \(\lambda \neq 1\). Then \(\mathcal{G}_\lambda\) is isomorphic to the Klein four-group \(V_4\), except in the following cases:
\begin{enumerate}[label=(\roman*), leftmargin=*, itemsep=1ex]
    \item If \(\mathrm{char}(\mathbb{K}) = 3\) and \(\lambda = -1\), then \(\mathcal{G}_\lambda\) is isomorphic to the symmetric group \(\mathfrak{S}_4\) of 24 elements.
    
    \item If \(\mathrm{char}(\mathbb{K}) = 2\), and the polynomial \(X^2 + X + 1\) splits in \(\mathbb{K}[X]\), and \(\lambda \in \{ j, j^2 \}\), then \(\mathcal{G}_\lambda\) is isomorphic to the alternating group \(\mathfrak{A}_4\) of order 12.
    
    \item If \(\mathrm{char}(\mathbb{K}) \neq 2\) and \(\mathrm{char}(\mathbb{K}) \neq 3\), and \(\lambda \in \{-1, 2, 1/2\}\), then \(\mathcal{G}_\lambda\) is isomorphic to the dihedral group \(D_4\) of 8 elements.
    
    \item If \(\mathrm{char}(\mathbb{K}) \neq 2\) and \(\mathrm{char}(\mathbb{K}) \neq 3\), and the polynomial \(X^2 + X + 1\) splits in \(\mathbb{K}[X]\), and \(\lambda \in \{ -j, -j^2 \}\), then \(\mathcal{G}_\lambda\) is isomorphic to \(\mathfrak{A}_4\).
\end{enumerate}
\end{thm}

\begin{corollary}\label{mainc}
    Let \( E \subset \mathbb{P}^1(\mathbb{K}) \) be a subset with four distinct elements. Then \( G_E \) is isomorphic to one of the following groups: 
    the Klein four-group \( V_4 \) of order 4, the dihedral group \( D_4 \) of order 8, the alternating group \(\mathfrak{A}_4\) of order 12, or the symmetric group \( \mathfrak{S}_4 \) of order 24. 
\end{corollary}

The following corollary applies to the field of rational numbers.

\begin{corollary} \label{coro:4.3.7}
Let \( E = \{x_1, x_2, x_3, x_4\} \), where \( x_1, x_2, x_3, \) and \( x_4 \) are four distinct rational numbers. If there exists a permutation \( (i, j, k) \in \{(1,2,3), (1,3,2), (3,2,1)\} \) such that
\[
(-2x_k + x_i + x_j)x_4 = 2x_i x_j - x_k(x_i + x_j),
\]
then the group \( G_E \) is isomorphic to the dihedral group \( D_4 \) of order 8. Otherwise, \( G_E \) is isomorphic to the Klein four-group \( V_4 \).
\end{corollary}

\section{Description of \texorpdfstring{$G_E$}{GE} for subsets \texorpdfstring{$E$}{E} with cardinality \texorpdfstring{$\leq 4$}{<= 4}}

In this section, we explore the existence and explicit construction of \( G_E \) associated with a given subset \( E \subseteq \mathbb{P}^1(\K) \) of cardinality at most 4. The following lemma characterizes homographies on the projective line \( \mathbb{P}^1(\K) \). A more general version of this result appears as Proposition 5.6 in \cite{michelle}.

\begin{lemma} \label{lemma1}
    Let \( x_1, x_2, x_3 \in \mathbb{P}^1(\K) \) and \( y_1, y_2, y_3 \in \mathbb{P}^1(\K) \) be two triples of distinct elements. Then there exists a unique homography \( h : \mathbb{P}^1(\K) \to \mathbb{P}^1(\K) \) such that \( h(x_i) = y_i \) for \( i = 1, 2, 3 \).
\end{lemma}

\begin{remark} \label{remk1}
Let \( E \subset \mathbb{P}^1(\mathbb{K}) \). If \( E \) contains exactly three elements, then \( G_E \cong \mathfrak{S}_3 \), the symmetric group on three elements, regardless of whether \( \mathbb{K} \) is a finite or infinite field.

If \( \mathbb{K} \) is infinite and \( |E| < 3 \), then \( G_E \) is infinite.

Now suppose \( \mathbb{K} \) is a finite field with \( q \) elements:

– If \( E \) consists of a single element, then \( G_E \) is in bijection with the set \( (\mathbb{K} \times \mathbb{K}) \setminus \Delta \), where \( \Delta \) is the diagonal of \( \mathbb{K} \times \mathbb{K} \). Hence, \( G_E \) has \( q^2 - q \) elements. For instance, if \( E = \{\infty\} \), there is a bijection that sends \( \sigma \in G_E \) to the pair \( (\sigma(0), \sigma(1)) \).

– If \( E \) contains two elements, then \( G_E \) is in bijection with the direct product \( (\mathbb{Z}/2\mathbb{Z}) \times (\mathbb{K} \setminus \{0\}) \), so \( G_E \) has \( 2(q - 1) \) elements. For instance, if \( E = \{\infty, 0\} \), there is a bijection that maps \( \sigma \in G_E \) to \( (\sigma|_E, \sigma(1)) \).
\end{remark}

\begin{exa} \label{example1}
For \( F = \{ \infty, 0, 1 \} \), we have
\[
G_F =\left\{ z,\ \ 1/z,\ \ 1-z,\ \ 1/(1-z),\ \ (z-1)/z, \ \ z/(z-1) \right\} .
\]
\end{exa}

\begin{lemma}\cite[Theorem 29]{samuel}\label{lemma2.2}
A homography \( h: \mathbb{P}^1(\K) \to \mathbb{P}^1(\K) \) is an involution if and only if there exists an element \( x \in \mathbb{P}^1(\K) \) such that \( h(x) \ne x \) and \( h^2(x) = x \).
\end{lemma}

\begin{proposition}\label{pro:4.3.5}
    For any set \( E \) of cardinality 4 in the projective line \( \mathbb{P}^1(\mathbb{K}) \), the group \( G_E \) contains a Klein four-group.
\end{proposition}

\begin{proof}
Assume that \( E = \{x_1, x_2, x_3, x_4\} \), where \( x_1, x_2, x_3, x_4 \) are distinct elements of \( \mathbb{K} \). We construct four distinct homographies that stabilize \( E \) and form a Klein four-group. For simplicity, let \( \beta = (x_1, x_2, x_3, x_4) \).

(a) The homography \( h_0 \) that maps \( \beta \) to itself is the identity homography, as it has more than two fixed points.

(b) According to Lemma~\ref{lemma1}, there exists a unique homography \( h_1 \) which satisfies \( h_1(x_1) = x_2 \), \( h_1(x_2) = x_1 \), and \( h_1(x_3) = x_4 \). Moreover, since \( h_1^2(x_1) = x_1 \), it follows from Lemma \ref{lemma2.2} that \( h_1 \) is an involution, hence \( h_1(x_4) = x_3 \). Therefore, \( h_1 \) maps \( \beta \) to \( (x_2, x_1, x_4, x_3) \).

(c) Similarly, there exists a homography \( h_2 \) that maps \( \beta \) to \( (x_3, x_4, x_1, x_2) \), with \( h_2(x_1) = x_3 \) and \( h_2^2 = \mathrm{id} \).

(d) By the same reasoning, there exists a homography \( h_3 \) mapping \( \beta \) to \( (x_4, x_3, x_2, x_1) \), such that \( h_3(x_1) = x_4 \) and \( h_3^2 = \mathrm{id} \).

By the definition of the homographies \( h_i \), for \( i = 0,1,2,3 \), we have \( h_3 = h_2 \circ h_1 = h_1 \circ h_2 \). Thus, the set
\[
J = \{ h_0, h_1, h_2, h_3 \}
\]
forms a Klein four-group.
\end{proof}

The homographies \( h_0, h_1, h_2, h_3 \) define bijective maps from \( E \) to \( E \). We can identify them with elements of the symmetric group \( \mathfrak{S}_4 \), with:
\[
h_0 = \text{Id}, \quad h_1 = (1 \ 2)(3 \ 4), \quad h_2 = (1 \ 3)(2 \ 4), \quad h_3 = (1 \ 4)(2 \ 3).
\]
As stated in \cite[Theorem 4.4.1]{alan}, the group \( J \) is a subgroup of \( \mathfrak{S}_4 \) with index 6. The representatives of the classes of \( \mathfrak{S}_4 / J \) are:
\[
\text{Id}, \quad (3 \ 4), \quad (2 \ 3), \quad (2 \ 4), \quad (2 \ 3 \ 4), \quad (2 \ 4 \ 3).
\]
\begin{defn}\label{defn:3.4.12}
Let \( x_1, x_2, x_3, x_4 \in \mathbb{P}^1(\mathbb{K}) \), where \( x_1, x_2, x_3 \) are distinct. Consider the unique homographic transformation \( \varphi \) on \( \mathbb{P}^1(\mathbb{K}) \) such that
\[
\varphi(x_1) = \infty, \quad \varphi(x_2) = 0, \quad \text{and} \quad \varphi(x_3) = 1.
\]
We define the cross-ratio of the quadruple \( (x_1, x_2, x_3, x_4) \) as the element \( \varphi(x_4) \in \mathbb{P}(\mathbb{K}) \), and denote it by
\(
[x_1, x_2, x_3, x_4].
\)
\end{defn}

If \( h \in \mathrm{PGL}_2(\mathbb{K}) \), then
\[
\varphi \circ h^{-1}(h(x_1)) = \infty, \quad \varphi \circ h^{-1}(h(x_2)) = 0, \quad \text{and} \quad \varphi \circ h^{-1}(h(x_3)) = 1,
\]
so that
\[
[x_1, x_2, x_3, x_4] = \varphi(x_4) = \varphi \circ h^{-1}(h(x_4)) = [h(x_1), h(x_2), h(x_3), h(x_4)].
\]
This means that the transformation \( h \) preserves the cross-ratio.

Furthermore, we still suppose \( x_1, x_2, x_3 \) are three distinct elements of \( \mathbb{K} \), and let \( \omega \in \mathbb{P}^1(\mathbb{K}) \). Consider the homography \( f \) defined by
\[
f(\omega) = 
\begin{cases}
 \dfrac{\omega - x_2}{\omega - x_1} \div  \dfrac{x_3 - x_2}{x_3 - x_1}, & \text{if } \omega \in \mathbb{K}, \\[1.5ex]
\dfrac{x_3 - x_1}{x_3 - x_2}, & \text{if } \omega = \infty.
\end{cases}
\]
Therefore, \( f \) satisfies:
\[
f(x_1) = \infty, \quad f(x_2) = 0, \quad f(x_3) = 1.
\]
By Definition~\ref{defn:3.4.12}, it follows that for all \( \omega \in \mathbb{P}^1(\mathbb{K}) \), we have:
\[
f(\omega) = [x_1, x_2, x_3, \omega].
\]

\begin{lemma}\cite[Proposition 6.2]{michelle} \label{lemma23}
Let \( x_{1}, x_{2}, x_{3}, x_{4} \in \mathbb{P}^1(\mathbb{K}) \) and \( y_{1}, y_{2}, y_{3}, y_{4} \in \mathbb{P}^1(\mathbb{K}) \) be four distinct elements in each set. Then, there exists a homography \( h \in \mathrm{PGL}_{2}(\mathbb{K}) \) such that \( h(x_i) = y_i \) for \( i = 1,2,3,4 \) if and only if
\[
[x_1, x_2, x_3, x_4] = [y_1, y_2, y_3, y_4].
\]
\end{lemma}

Moreover, as noted in Section~2.2 of \cite{samuel}, given four distinct points \( x_1, x_2, x_3, x_4 \in \mathbb{P}^1(\mathbb{K}) \), and a permutation \( \sigma \) of \(\{1,2,3,4\}\), the cross-ratio \( [x_{\sigma(1)}, x_{\sigma(2)}, x_{\sigma(3)}, x_{\sigma(4)}] \) is uniquely determined by \( [x_1, x_2, x_3, x_4] \), the characteristic of \( \mathbb{K} \), and the permutation \( \sigma \).  
Therefore, there are six equivalence classes of \(\mathfrak{S}_4\) modulo \(J\). The following proposition provides further details on the required conditions.

\begin{figure}[H]
    \centering
    \includegraphics[width=0.9\linewidth]{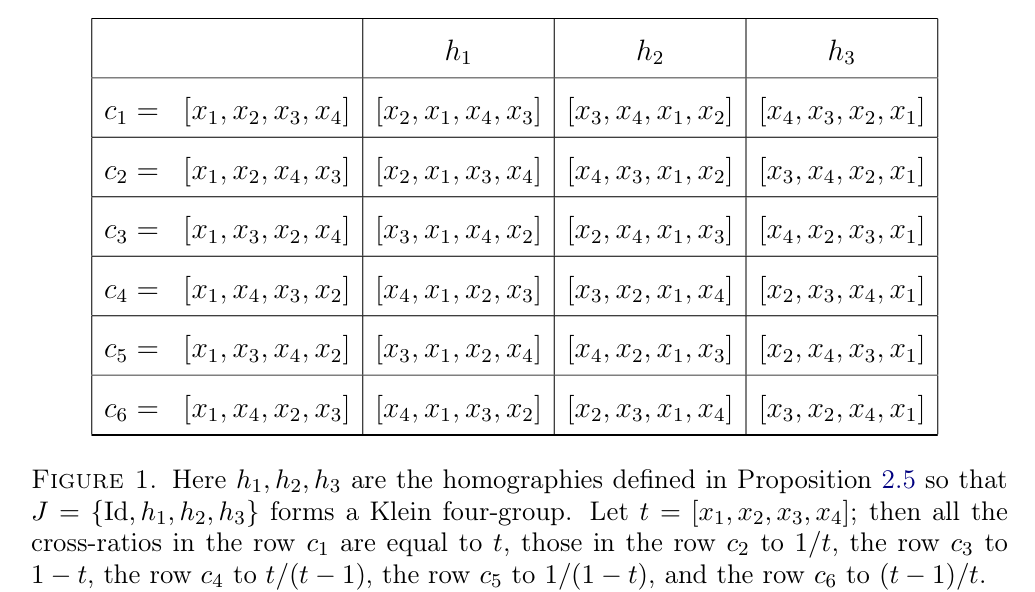}
\end{figure}

Notice that if we still consider \(x_1, x_2, x_3, x_4\) to be four distinct elements of \(\mathbb{P}^1(\K)\), then according to row \(c_1\) of Figure~1, we have the equality:
\[
[x_1,x_2,x_3,x_4] = [x_2,x_1,x_4,x_3] = [x_3,x_4,x_1,x_2] = [x_4,x_3,x_2,x_1].
\]
Therefore, by Lemma~\ref{lemma23}, there exist three distinct and non-trivial homographies that respectively map the quadruple \((x_1,x_2,x_3,x_4)\) to \((x_2,x_1,x_4,x_3)\), \((x_3,x_4,x_1,x_2)\), and  \((x_4,x_3,x_2,x_1)\).
These correspond respectively to the homographies \(h_1\), \(h_2\), and \(h_3\) introduced in Proposition~\ref{pro:4.3.5}. This observation offers an alternative proof for Proposition~\ref{pro:4.3.5}. Also, from Lemma~\ref{lemma23}, one deduces that the index of \(\mathcal{G}_{\lambda}\) in the symmetric
 group \(\mathfrak{S}_4\) is the number of distinct elements in \( \{c_1, \cdots
 ,c_6\} \) when \((x_1,x_2,x_3,x_4)= (\infty, 0, 1, \lambda)\).\\
The following proposition complements Theorem 26 of \cite{samuel} by providing further details about the exception mentioned.

\begin{proposition} \label{prop22}
Let \(\mathbb{K}\) be a field different from \(\mathbb{F}_2\), and let \(\lambda \in \mathbb{K}\) such that \(\lambda \neq 0\) and \(\lambda \neq 1\). Setting \((x_1, x_2, x_3, x_4) = (\infty, 0, 1, \lambda)\), and following the notation in Figure 1, we distinguish the following cases:
\begin{enumerate}[label=(\roman*), leftmargin=*, itemsep=1ex]
  \item If \(\mathrm{char}(\mathbb{K}) = 3\), then \(\lambda = -1\) iff \(c_1 = c_2 = c_3 = c_4 = c_5 = c_6\).

  \item If \(\mathrm{char}(\mathbb{K}) = 2\), and the polynomial \(X^2 + X + 1\) splits in \(\mathbb{K}[X]\), then \(\lambda \in \{ j, j^2 \}\) iff \(c_1 = c_5 = c_6\).

  \item If \(\mathrm{char}(\mathbb{K}) \neq 2\) and \(\mathrm{char}(\mathbb{K}) \neq 3\), and the polynomial \(X^2 + X + 1\) does not split in \(\mathbb{K}[X]\), then \(\lambda \in \{ -1, 2, 1/2 \}\) iff \(c_1 = c_2\), or \(c_1 = c_3\), or \(c_1 = c_4\).

  \item If \(\mathrm{char}(\mathbb{K}) \neq 2\) and \(\mathrm{char}(\mathbb{K}) \neq 3\), and the polynomial \(X^2 + X + 1\) splits in \(\mathbb{K}[X]\), then \(\lambda \in \{ -1, 1/2, 2, -j, -j^2 \}\) iff \(c_1 = c_2\), or \(c_1 = c_3\), or \(c_1 = c_4\), or \(c_1 = c_5 = c_6\).
\end{enumerate}
In all other cases, the 6 elements \(c_1 \cdots, c_6\) are pairwise distinct.
\end{proposition}

\begin{proof}
It suffices to solve the following equations in \(\mathbb{K}- \{0, 1\}\);
\begin{equation} \label{eqt}
    t = \frac{1}{t}, \quad t = 1 - t, \quad t = \frac{t}{t-1}, \quad t = \frac{1}{1-t}, \quad t = \frac{t-1}{t}.
\end{equation}

The solutions to these equations depend on the characteristic of \(\mathbb{K}\):

\begin{enumerate}[label=, leftmargin=*, itemsep=1ex]
    \item If \(\mathrm{char}(\mathbb{K}) = 3\), then all the equations in~\eqref{eqt} have \(\{-1\}\) as solution. This justifies case~(i).

    \item If \(\mathrm{char}(\mathbb{K}) = 2\), then only the fourth and fifth equations from~\eqref{eqt} have solutions, which are for both cases \(\{j, j^2\}\). If \(\lambda \in \{ j, j^2 \} \), then \( c_1 = c_5 = c_6 = \lambda \), and \( c_2 = c_3 = c_4 = \lambda^2 \).  This justifies case~(ii).

    \item If \(\mathrm{char}(\mathbb{K}) \neq 2\) and \(\mathrm{char}(\mathbb{K}) \neq 3\), and the polynomial \(X^2 + X + 1\) does not split in \(\mathbb{K}[X]\), then only the first, second, and third equations have solutions, which are distinct and are respectively:
    \[
    \{-1\}, \quad \left\{ \frac{1}{2} \right\}, \quad \{2\}.
    \]
     \textit{If} \( \lambda = -1 \), then \( c_1 = c_2 = -1,\, c_3 = c_6 = 2,\, c_4 = c_5 = 1/2 \); \\
    \textit{If} \( \lambda = 1/2 \), then \( c_1 = c_3 = 1/2,\, c_2 = c_5 = 2,\, c_4 = c_6 = -1 \);\\
  \textit{If} \( \lambda = 2 \), then \( c_1 = c_4 = 2,\, c_3 = c_5 = -1,\, c_2 = c_6 = 1/2 \).\\
 This corresponds to case~(iii).

    \item If \(\mathrm{char}(\mathbb{K}) \neq 2\) and \(\mathrm{char}(\mathbb{K}) \neq 3\), and the polynomial \(X^2 + X + 1\) splits in \(\mathbb{K}[X]\), then all the elements \(-1, 2, 1/2, -j, -j^2\) are distinct and all equations in~\eqref{eqt} have solutions, which are respectively:
    \[
    \{-1\}, \quad \left\{ \frac{1}{2} \right\}, \quad \{2\}, \quad \{-j, -j^2\}, \quad \{-j, -j^2\}.
    \] \textit{If} \( \lambda \in \{-j,\,-j^2\} \), then \( c_1 = c_5 = c_6 = \lambda,\, c_2 = c_3 = c_4 = 1/\lambda \). This justifies case~(iv).
\end{enumerate}
\end{proof}

Now we have all the necessary tools to prove Theorem~\ref{main}.
\begin{proof}[Proof of Theorem~\ref{main}]
Let \(E = \{x_1, x_2, x_3, x_4\}\), with \(x_1 = \infty\), \(x_2 = 0\), \(x_3 = 1\), and \(x_4 = \lambda\). According to Proposition~\ref{pro:4.3.5}, the group \(G_E\) (also denoted \(\mathcal{G}_\lambda\)) contains a Klein four-group as described in the first row of Figure~1. 

If we assume that \(G_E\) is larger than the Klein four-group, then by Lemma~\ref{lemma23}, there exists a permutation \(\sigma\) of \(\{1, 2, 3, 4\}\) such that the cross-ratio \([x_{\sigma(1)}, x_{\sigma(2)}, x_{\sigma(3)}, x_{\sigma(4)}]\) is different from the four permutations in the row \(c_1\) of Figure~1, yet satisfies
\[
[x_{\sigma(1)}, x_{\sigma(2)}, x_{\sigma(3)}, x_{\sigma(4)}] = [x_1, x_2, x_3, x_4].
\]
This occurs only under the conditions stated in Proposition~\ref{prop22}:

\begin{enumerate}[label=(\roman*), leftmargin=*, itemsep=1ex]

 \item[(i)] If \( \mathrm{char}(\mathbb{K}) = 3 \) and \( \lambda = -1 \), the 24 cross ratios in Figure 1 are equal. By Lemma 2.7 the group \( \mathcal{G}_\lambda \) contains 24 elements, hence is the full symmetric group \( \mathfrak{S}_4 \) of order 24.
  
  \item[(ii)] If \( \mathrm{char}(\mathbb{K}) = 2 \) and the polynomial \( X^2 + X + 1 \) splits in \( \mathbb{K}[X] \), and \( \lambda \in \{j, j^2\} \), then the set of cross ratios in Figure 1 contains exactly two distinct elements, namely \( j \) and \( j^2 \), hence by Lemma~\ref{lemma23} the group \( \mathcal{G}_\lambda \) is a subgroup of index 2 in \( \mathfrak{S}_4 \) and therefore it is the alternating group \( \mathfrak{A}_4 \) of order 12.
  
  \item[(iii)] If \( \mathrm{char}(\mathbb{K}) \notin \{2,3\} \) and \( \lambda \in \{-1, 2, 1/2\} \), then the set of cross ratios in Figure 1 contains exactly 3 distinct elements, namely \( -1, 2, 1/2 \), hence the group \( \mathcal{G}_\lambda \) is a subgroup of index 3 in \( \mathfrak{S}_4 \). It follows that \( \mathcal{G}_\lambda \) is a dihedral group \( D_4 \) of order 8.
  
  \item[(iv)] If \( \mathrm{char}(\mathbb{K}) \notin \{2,3\} \), and the polynomial \( X^2 + X + 1 \) splits in \( \mathbb{K} \) and \( \lambda \in \{-j, -j^2\} \), then the set of cross ratios in Figure 1 contains exactly two distinct elements, namely \( -j \) and \( -j^2 \), hence \( \mathcal{G}_\lambda \) is again the alternating group \( \mathfrak{A}_4 \) of order 12.

\end{enumerate}
\end{proof}

\begin{lemma}\label{lemma2} 
    Let \(E\) be a subset of the projective line \(\mathbb{P}^1(\K)\) consisting of 3 distinct elements. Then, the group \(G_E\) is conjugate to the group stabilizing the set \(\{\infty, 0, 1\}\).
\end{lemma}

\begin{proof}
Suppose that \(E = \{x_1, x_2, x_3\} \subset \mathbb{P}^1(\K)\). Let \(h\) be an element of \(G_F\), where \(F\) is defined as in Example~\ref{example1}. By Lemma~\ref{lemma1}, there exists a unique homography \(f\) such that
\[
f(x_1) = \infty, \quad f(x_2) = 0, \quad f(x_3) = 1.
\]
Define the homography \(j_h = f^{-1} \circ h \circ f\). Then \(j_h \in G_E\), and this construction implies that
\[
G_E = \{ j_h \mid h \in G_F \}.
\]
\end{proof}

\begin{lemma}\label{lemma222} 
Let \(\K\) be a field different from \(\mathbb{F}_2\), and let \(E\) be a subset of the projective line \(\mathbb{P}^1(\K)\) consisting of 4 distinct elements. Then, there exists an element \(\lambda \in \K -\{0,1 \}\) such that the group \(G_E\) is conjugate to the group stabilizing the set \(\{\infty, 0, 1, \lambda\}\).
\end{lemma}

\begin{proof}
Suppose that \(E = \{x_1, x_2, x_3, x_4\}\), and let \(f\) be the homography such that
\[
f(x_1) = \infty, \quad f(x_2) = 0, \quad f(x_3) = 1.
\]
Set \(f(x_4) = \lambda\), and consider the group \(\mathcal{G}_\lambda\) as in Theorem~\ref{main}. For any \(g \in \mathcal{G}_\lambda\), define \(j_g = f^{-1} \circ g \circ f\). Then \(j_g \in G_E\), and we have
\[
G_E = \{ j_g \mid g \in \mathcal{G}_\lambda \}.
\]
\end{proof}

\begin{proof}[Proof of Corollary \ref{mainc}]
The result is an immediate consequence of Theorem~\ref{main} together with Lemma~\ref{lemma2}.
\end{proof}

\begin{proof}[Proof of Corollary \ref{coro:4.3.7}]
This is a direct application of Theorem \ref{main} and Lemma \ref{lemma222}. Observe that \(\mathbb{Q}\) has characteristic zero and the polynomial \(X^2 + X+ 1\) does not split in $\mathbb{Q}$. Therefore, according to Theorem \ref{main} and Definition \ref{defn:3.4.12}, we have that
\[
\left(\dfrac{x_4 - x_2}{x_4 - x_1}\right) \div \left(\dfrac{x_3 - x_2}{x_3 - x_1}\right)
\]
is equal to either \(-1\), \(2\), or \(1/2\). This corresponds to the following identity:
\[
(-2x_k + x_i + x_j)x_4 = 2x_i x_j - x_k(x_i + x_j),
\]
where \((i,j,k)\) is equal to \((1,2,3)\), \((1,3,2)\), or \((3,2,1)\), respectively.
\end{proof}

\noindent\textbf{Acknowledgements:} The author expresses heartfelt gratitude to Pr. Michel Waldschmidt for his guidance at every stage of this project.

\end{document}